\documentclass[a4paper, 12pt]{article}

\usepackage[utf8]{inputenc}
\usepackage[T1]{fontenc}
\usepackage[a4paper, margin=2.5cm]{geometry}
\usepackage{needspace}
\usepackage{xcolor}

\usepackage{libertine} 
\usepackage{inconsolata} 

\usepackage{amsmath, amsthm, amssymb}
\usepackage{graphicx}
\usepackage{enumerate}
\usepackage{authblk}
\usepackage{thmtools}
\usepackage{thm-restate}
\usepackage[colorlinks=true, citecolor=red]{hyperref}
\makeatletter 
\g@addto@macro{\UrlBreaks}{\UrlOrds}
\makeatother
\usepackage{float}

\renewenvironment{abstract}
{\small\vspace{-1em}
\begin{center}
\bfseries\abstractname\vspace{-.5em}\vspace{0pt}
\end{center}
\list{}{
\setlength{\leftmargin}{0.6in}
\setlength{\rightmargin}{\leftmargin}}
\item\relax}
{\endlist}

\declaretheorem[name=Theorem, numberwithin=section]{theorem}
\declaretheorem[name=Lemma, sibling=theorem]{lemma}

\declaretheorem[name=Corollary, sibling=theorem]{corollary}

\declaretheorem[name=Problem, sibling=theorem]{problem}
\declaretheorem[name=Claim, sibling=theorem]{claim}

\declaretheorem[name=Observation, style=remark, sibling=theorem]{observation}

\def\cqedsymbol{\ifmmode$\lrcorner$\else{\unskip\nobreak\hfil
\penalty50\hskip1em\null\nobreak\hfil$\lrcorner$
\parfillskip=0pt\finalhyphendemerits=0\endgraf}\fi}

\def\N{\mathbb{N}}

\def\R{\mathbb{R}}

\def\Z{\mathbb{Z}}

\let\le\leqslant
\let\ge\geqslant
\let\leq\leqslant
\let\geq\geqslant

\title{Optimal labelling schemes for adjacency, comparability, and reachability
\footnotetext{\llap{\textsuperscript{*}}M.B. and L.E. are supported by the ANR Projects DISTANCIA (\textsc{ANR-17-CE40-0015}) and GrR (\textsc{ANR-18-CE40-0032}), and by LabEx PERSYVAL-lab (\textsc{ANR-11-LABX-0025}).}}

\author[1]{Marthe Bonamy*\textsuperscript{,}}
\author[2]{Louis Esperet*\textsuperscript{,}}
\author[3]{Carla Groenland}
\author[3]{Alex Scott}
\affil[1]{CNRS, LaBRI, Université de Bordeaux, Bordeaux, France.}
\affil[2]{CNRS, G-SCOP, Univ. Grenoble Alpes, Grenoble, France.}
\affil[3]{Mathematical Institute, University of Oxford, Oxford OX2 6GG, United Kingdom.}
\date{\today}

\begin{document}
\maketitle

\begin{abstract}
We construct asymptotically optimal adjacency labelling schemes for every hereditary class containing $2^{\Omega(n^2)}$ $n$-vertex graphs as $n\to \infty$. This regime contains many classes of interest, for instance perfect graphs or comparability graphs, for which we obtain an adjacency labelling scheme with labels of $n/4+o(n)$ bits per vertex. This implies the existence of a reachability labelling scheme for digraphs with labels of $n/4+o(n)$ bits per vertex and comparability labelling scheme for posets with labels of $n/4+o(n)$ bits per element. All these results are best possible, up to the lower order term.
\end{abstract}

\section{Introduction}\label{sec:intro}

When representing graphs, say with adjacency lists or
matrices, vertex identifiers are usually just pointers in the data structure.  
In contrast, a graph is \emph{implicitly   represented} when each vertex of the graph carries enough information
so that some properties, for instance adjacency, can be efficiently determined from the identifiers in a local manner (cf.~\cite{KNR88,Spinrad03}). 
The standard example is that of interval graphs: if $G$ is an interval graph with $n$ vertices\footnote{Throughout the paper, $n$ is implicitly the number of vertices in the graph at hand.}, we can assign to each vertex $u$ some interval $I(u) \subseteq [1,2n]$ with integer endpoints so
that $u,v$ are adjacent if and only if $I(u) \cap I(v) \neq \emptyset$. 
Clearly, to represent $G$ it suffices to store the corresponding intervals. Although $G$ may have a quadratic number of edges, such an implicit representation uses 
$2\log{n} + O(1)$ bits per vertex\footnote{Throughout the paper, $\log{n}$ denotes the binary logarithm of $n$.}.  
Compact representations have several advantages, not only for memory storage, but also from algorithmic perspectives. 

The route to more efficient representations is conceptually simple: it is a matter of exploiting the structure of the graph class to spell out as few adjacencies as possible and derive the rest from the existing information. On the other hand, how can we argue that a representation is best possible?

Let us define the problem more formally. We say that a graph class $\mathcal{G}$ admits an \emph{adjacency labelling scheme} with labels of $f(n)$ bits if, for any integer $n$, we can label the vertices of any $n$-vertex graph of $\mathcal{G}$ with strings of $f(n)$ bits such that, using only the labels assigned to $u$ and $v$, we can determine whether $u$ and $v$ are adjacent in the graph. An adjacency labelling scheme is \textit{efficient} if the encoding takes polynomial time and the decoding takes constant time in the word RAM model with words of $\Theta(\log n)$ bits (see Section \ref{subsec:dict} for the details). 
An adjacency labelling scheme describing a graph $G$ also describes all the induced subgraphs of $G$, so it is customary to consider adjacency labelling schemes for \emph{hereditary} classes of graphs, that is, classes of graphs that are closed under taking induced subgraphs.

Given a graph class $\mathcal{G}$ and an integer $n$, the set of $n$-vertex graphs\footnote{All the graphs considered in this paper are unlabelled, so all the graphs are considered up to isomorphism.} of $\mathcal{G}$ is denoted by $\mathcal{G}_n$. If $\mathcal{G}$ admits an adjacency labelling scheme with labels of $f(n)$ bits, then the total number of bits describing an $n$-vertex graph $G\in \mathcal{G}$ is $n\cdot f(n)$, and thus $|\mathcal{G}_n|\le 2^{n\cdot f(n)}$. Given a class $\mathcal{G}$ and an integer $n$, let 
$\mu_\mathcal{G}(n)=\tfrac1{n} \log |\mathcal{G}_n|$ (we will write  $\mu(n)$, if $\mathcal{G}$ is clear from context).
Note that $\mu_\mathcal{G}(n)\leq \frac{n}2$ for any $\mathcal{G}$. The observation above can be restated as follows.

\begin{observation}\label{obs:mu}
For every class $\mathcal{G}$, no adjacency labelling scheme for $\mathcal{G}$ has less than $\mu_{\mathcal{G}}(n)$ bits per vertex.
\end{observation}

 A natural question is whether this lower bound of $\mu(n)$ can be attained. There are examples of non-hereditary classes for which this lower bound cannot be attained~\cite{Spinrad03}, so in the remainder of the paper we only consider hereditary classes.
 
 \smallskip

In the regime $\mu(n)=\Theta(\log n)$, the \emph{Implicit Graph Conjecture}~\cite{KNR88,Spinrad03} posits that every hereditary class $\mathcal{G}$ with $\mu_{\mathcal{G}}(n)=O(\log n)$ has an adjacency labelling scheme with labels of $O(\log n)=O(\mu(n))$ bits per vertex. Although the conjecture has been proved in some special cases, it is still open in general (see~\cite{ACLZ15,BGKTW21} for recent results and references on the conjecture). Note that in this regime the conjecture only posits that there is an adjacency labelling scheme with labels of at most a constant times $\mu(n)$ bits per vertex.

\smallskip

The purpose of our paper is to show that the lower bound $\mu(n)$ is attained (up to lower order terms) in the denser regime $\mu(n)=\Theta(n)$, which also has many applications.
We prove the following.

\begin{theorem}
\label{thm:Hfree}
Let $\mathcal{G}$ be a hereditary graph class. For each $\delta>0$, $\mathcal{G}$ has an efficient adjacency labelling scheme using $\mu_{\mathcal{G}}(n)+\delta n$ bits per vertex. Moreover, this can be turned into an adjacency labelling scheme using $\mu_{\mathcal{G}}(n)+o(n)$ bits per vertex, which is tight up to the $o(n)$ term.
\end{theorem}
Theorem~\ref{thm:Hfree} shows that in this regime,  graphs from $\mathcal{G}$ can not only be compressed in an optimal way, but that the representation can be uniformly distributed among the vertices, with efficient encoding and decoding. 
To the best of our knowledge, this was already known only in the specific cases where $\mathcal{G}$ is the class of all graphs, or the class of all bipartite graphs (see~\cite{Alon17,AlstrupKaplanThorupZwick} for recent sharper results on these two classes).

\medskip

Note that Theorem~\ref{thm:Hfree} is applicable to a wide range of graph classes, for instance classes of graphs of bounded chromatic number (such as bipartite graphs), or classes of graphs excluding some (induced) subgraphs (such as perfect graphs, chordal graphs, or split graphs). An example is the class of \emph{string graphs on $S$} for some surface $S$, i.e.\ the class of intersection graphs of continuous curves embedded in $S$. It follows from \cite[Corollary 3]{AlonBaloghBollobasMorris} that for this class $\mu(n)=(1-1/\omega(S)+o(1))\cdot \tfrac{n}2$, where $\omega(S)$ denotes the size of the largest complete graph embeddable in $S$. Using Theorem~\ref{thm:Hfree}, this directly implies the following.

\begin{corollary}
\label{cor:string}
For any surface $S$, the class of string graphs on $S$  has an adjacency labelling scheme using $(1-1/\omega(S))\cdot \tfrac{n}2+o(n)$ bits per vertex, which is asymptotically tight.
\end{corollary}

Another interesting example for us is the class of graphs excluding the 5-cycle $C_5$ as an induced subgraph. It is well known that for this class, we have $\mu(n)=\tfrac{n}4+o(n)$ (more will be said about this in Section~\ref{sec:hereditary}). We thus obtain the following as an immediate consequence of Theorem~\ref{thm:Hfree}.

\begin{corollary}
\label{cor:perfect}
The class of graphs with no induced $C_5$ has an adjacency labelling scheme using $n/4+o(n)$ bits per vertex, which is asymptotically tight.
\end{corollary}

Given a partially ordered set (poset, in short) $(P,<)$, the \emph{comparability graph} of $P$ is the graph whose vertices are the elements of $P$, and in which two vertices are adjacent if and only if the two corresponding elements of $P$ are comparable. A graph $G$ is \emph{perfect} if for any induced subgraph $H$ of $G$, the chromatic number and the clique number of $H$ coincide. It is well known that comparability graphs are perfect.
Note that the class of graphs with no induced $C_5$ contains the class of perfect graphs, and thus the class of comparability graphs, hence Corollary~\ref{cor:perfect} also applies to these classes (and is tight for them as well).

\medskip

The key ingredient in the proof of Theorem \ref{thm:Hfree} is Szemerédi's Regularity Lemma, which we describe in Section~\ref{sec:prelims}. The downside of this tool is that even though our results are best possible asymptotically, the $o(n)$ term in Theorem~\ref{thm:Hfree} is such that $\tfrac1{n}\cdot o(n)$ is an extremely slowly decreasing function of $n$. Moreover, the ``efficient" algorithm hides huge constants for encoding and decoding. In the case of comparability graphs (for which we present two important applications below), we can significantly improve the lower order term and the complexity of our encoding and decoding compared to Corollary \ref{cor:perfect}. The proof is short and does not use any deep result.

\begin{theorem}
\label{thm:comparability}
The class of comparability graphs admits an adjacency scheme with labels of $n/4+O(n^{3/4}\log^2 n)$ bits per vertex and an efficient adjacency labelling scheme with
labels of $n/4+O(n(\log \log n)^2/\log^{1/4}n)$ bits per vertex.
\end{theorem}

We now describe two important consequences of Theorem~\ref{thm:comparability} for comparability labelling schemes in posets and reachability labelling schemes in digraphs.

\subsubsection*{Comparability in posets}

We say that a class $\mathcal{P}$ of posets admits a \emph{comparability labelling scheme} with labels of $f(n)$ bits if for every integer $n$, we can label the elements of any $n$-element poset of $\mathcal{P}$ with strings of at most $f(n)$ bits such that it can be determined whether $x$ and $y$ are comparable in the poset (and if so, whether $x \le y$) using only the labels assigned to $x$ and $y$. 

Given a poset $(P,<)$, we can consider a linear ordering $x_1,\ldots,x_n$ of $P$ (i.e.\ an ordering such that $i<j$ whenever $x_i<x_j$). If each element $x_i$ stores its index $i$ (this costs $O(\log n)$ bits), then whenever two elements $x$ and $y$ are comparable they can decide whether $x<y$ or $y<x$ by looking at their indices. Consequently, by appending indices of the elements to an adjacency labelling scheme for the comparability graphs of the posets in a class $\mathcal{P}$, we obtain a comparability labelling scheme for the class $\mathcal{P}$, with only $O(\log n)$ additional bits per element. We thus obtain the following as an immediate consequence of Theorem~\ref{thm:comparability}.

\begin{corollary}
\label{cor:comparabilityposet}
The class of all posets admits a comparability labelling scheme with labels of $n/4+O(n^{3/4}\log^2 n)$ bits per element and an efficient comparability labelling scheme with
labels of $n/4+O(n(\log \log n)^2/\log^{1/4}n)$ bits per element.
\end{corollary}

Munro and Nicholson \cite{MunroNicholson} proved that posets can be represented in $n^2/4+o(n^2)$ bits in such a way that queries of the form `$a\leq b?$' can be answered by inspecting only a constant number of bits. 
Corollary~\ref{cor:comparabilityposet} shows that an encoding of the same total size can be obtained by distributing the information uniformly between the elements of the poset.

\subsubsection*{Reachability in digraphs}

We say that a vertex $u$ can \emph{reach} a vertex $v$ in a directed graph (digraph, in short) if there is a directed path in the digraph from $u$ to $v$. 
We say that a class $\mathcal{C}$ of digraphs admits a \emph{reachability labelling scheme} with labels of $f(n)$ bits if for any integer $n$, we can label the vertices of any $n$-vertex digraph of $\mathcal{C}$ with strings of at most $f(n)$ bits such that it can be determined whether $u$ can reach $v$ in the digraph using only the labels assigned to $u$ and $v$.

It is well-known (see for instance~\cite{DulekebaGawrychowskiJanczewski}) that reachability queries in digraphs can be reduced to comparability queries in posets as follows. Given a digraph $D$, we start by contracting each strong component of $D$ into a single vertex. Let $D'$ be the resulting acyclic digraph, and suppose that ayclic digraphs have a reachability labelling scheme with labels of $f(n)$ bits. Then a reachability labelling for $D'$ can be turned into a reachability labelling for $D$ by giving to each vertex $v$ of $D$ the labelling in $D'$ of the strong component containing $v$, followed by a $O(\log n)$ bit label uniquely identifying each strong component. This gives a reachability labelling scheme for all digraphs with labels of $f(n)+O(\log n)$ bits per vertex.

So it suffices to design a reachability labelling for acyclic digraphs. Given an acyclic digraph $D$, and two vertices $u,v$ in $D$, we write $u<v$ if there is a directed path from $u$ to $v$. Since $D$ is acyclic, $(D,<)$ forms a poset and comparability queries in this poset are precisely reachability queries in $D$. This immediately implies the following.

\begin{corollary}
\label{cor:digraphs_reachability}
The class of all digraphs admits a reachability labelling scheme with labels of $n/4+O(n^{3/4}\log^2 n)$ bits per vertex and an efficient reachability labelling scheme with
labels of $n/4+O\left(\tfrac{n(\log \log n)^2}{\log^{1/4}n}\right)$ bits per vertex.
\end{corollary}

This improves upon a recent result by Dul{\k{e}}ba, Gawrychowski and Janczewski~\cite{DulekebaGawrychowskiJanczewski}, who proved that digraphs admit an efficient reachability labelling scheme with labels of size $n/3+o(n)$.

\subsubsection*{Induced-universal graphs}

Forgoing all complexity concerns, compact labelling (resp. reachability, comparability) schemes for a class of graphs can be seen from a purely structural perspective, which is that of universal graphs.

Given a graph class $\mathcal{G}$, we say that $\mathcal{G}$ admits \emph{induced-universal graphs} on $f(n)$ vertices if for every $n$, there is a graph on $f(n)$ vertices that contains every $n$-vertex element of $\mathcal{G}$ as an induced subgraph. 

As observed in~\cite{KNR88,KNR92}, admitting universal graphs on $f(n)$ vertices is equivalent to admitting an adjacency labelling scheme using $\log(f(n))$ bits\footnote{Here we need to add the condition that the encoding function is injective, which in general only costs $O(\log n)$ additional bits per vertex and can thus be included in the lower order term.}. 
Indeed, to label an $n$-vertex element of $\mathcal{G}$, it suffices to embed it in the universal graph on $f(n)$ vertices as an induced subgraph, and label its vertices with their corresponding names in the universal graph. Conversely, given an adjacency labelling scheme using $h(n)$ bits, we define a universal graph on $2^{h(n)}$ vertices as follows. Let all possible labels on $h(n)$ bits form its vertex set, and let the edges be defined by the labelling scheme applied to each pair of labels. The resulting graph has $2^{\log (f(n))}$ vertices and contains every $n$-vertex element of $\mathcal{G}$ as induced subgraph.

Theorem~\ref{thm:Hfree} therefore has the following immediate transposition in the realm of induced-universal graphs. 

\begin{corollary}
\label{cor:Hfree-univ}
Let $\mathcal{G}$ be a hereditary graph class. Then for any integer $n$ there is a graph $G_n$ on $2^{\mu_{\mathcal{G}}(n)+o(n)}$ vertices containing all $n$-vertex graphs of $\mathcal{G}$ as induced subgraph. This is optimal up to the lower order term. 
\end{corollary}

\paragraph{Organisation of the paper}
We outline the results from the literature that we will need in Section~\ref{sec:prelims}. We make use of near-optimal dictionaries, for which we state the required results in Section~\ref{subsec:dict}. We prove Theorem \ref{thm:Hfree} in Section~\ref{sec:Hfree}.
In Section~\ref{sec:rl:digraphs} we give a labelling scheme for comparability graphs with a trade-off between the number of bits used for the labels and the decoding time and then deduce Theorem \ref{thm:comparability}.
We conclude with some remarks and open problems in Section~\ref{sec:ccl}.

\section{Preliminaries}
\label{sec:prelims}
In the section, we state the auxiliary results from the literature that we need for the proofs of Theorem \ref{thm:Hfree}.
and Theorem \ref{thm:comparability}.

\subsection{Hereditary Classes}\label{sec:hereditary}
Let $\mathcal{H}(a,b)$ be the set of graphs whose vertex set can be partitioned into $a$ cliques and $b$ independent sets. The \textit{colouring number} $\chi_c(\mathcal{G})\in \N\cup \{\infty\}$ of a hereditary graph class $\mathcal{G}$ is the supremum of the integers $r\in \N$ for which there exist $a,b\in \mathbb{Z}_{\geq 0}$ with $a+b=r$ such that $\mathcal{H}(a,b)\subseteq \mathcal{G}$ (it can easily be checked that $\chi_c(\mathcal{G})$ is finite if and only if $\mathcal{G}$ is not the class of all graphs). The following was proved by Alekseev~\cite{alekseev93} and Bollob\'as and Thomason~\cite{BT95,BT97}.

\begin{theorem}[Alekseev-Bollob\'as-Thomason~\cite{alekseev93,BT95,BT97}]\label{thm:ABT}
Let $\mathcal{G}$ be a hereditary class of graphs with $r=\chi_c(\mathcal{G})\ge 1$. Then $$|\mathcal{G}_n|=2^{(1-1/r+o(1))n^2/2}.$$ 
\end{theorem}

Note that if a hereditary class $\mathcal{G}$ contains all bipartite graphs or all split graphs, then $\chi_c(\mathcal{G})\ge 2$ and thus $|\mathcal{G}_n|\ge 2^{(1/4+o(1))n^2}$ and $\mu_\mathcal{G}(n)\ge n/4+o(n)$ by Theorem~\ref{thm:ABT}. This applies in particular to the class of comparability graphs, and to any class of graphs excluding some non-bipartite graph (such as $C_5$) as an induced subgraph.

On the other hand, observe that for any $a,b\in \mathbb{Z}_{\geq 0}$ with $a+b=3$, the 5-cycle $C_5$ is contained in $\mathcal{H}(a,b)$. This shows that if $\mathcal{G}$ is the class of graphs with no induced $C_5$, then $\chi_c(\mathcal{G})<3$, and thus $\chi_c(\mathcal{G})=2$ by the paragraph above. By Theorem~\ref{thm:ABT}, and the paragraph above this implies that $|\mathcal{G}_n|=  2^{(1/4+o(1))n^2}$ and $\mu_\mathcal{G}(n)= n/4+o(n)$, and the same applies to the class of comparability graphs.

\subsection{Regularity}

For a graph $G$ and two disjoint subsets of vertices $A, B\subseteq V(G)$, let $e(A,B)$ denote the number of edges of $G$ with an endpoint in $A$ and an endpoint in $B$. If $A$ and $B$ are non-empty, define the \textit{density} of edges between $A$ and $B$ by $d(A, B) = \frac{e(A,B)}{|A||B|}$. For $\varepsilon > 0$, the
pair $(A, B)$ is called \textit{$\varepsilon$-regular} if $|d(A, B) - d(X, Y )| < \varepsilon$ for every $X\subseteq A$ and $Y\subseteq B$ with $|X| \geq \varepsilon |A|$ and $|Y | \geq \varepsilon|B|$. 

An \textit{$\varepsilon$-regular partition} of the set of vertices $V$ of a graph $G$ is a partition of $V$ into pairwise disjoint vertex sets $V_0,\dots,V_t$, for some $t$, such that
\begin{enumerate}
    \item $|V_0|\leq \varepsilon |V|$;
    \item $|V_1|=\dots =|V_t|$;
    \item all but at most $\varepsilon \binom{t}2$ of the pairs $(V_i,V_j)$ for $1\leq i<j\leq t$ are $\varepsilon$-regular.
\end{enumerate}
The fact that such partitions exist with a number of parts depending only on $\varepsilon$ (and not on the size of the graph) is the well-known Szemer\'{e}di
Regularity Lemma. We will need an efficient constructive version of this lemma. 
\begin{lemma}[Algorithmic version of Szemer\'{e}di Regularity Lemma~\cite{AlonDukeLefmannRodlYuster}]
\label{lem:rl:alg_sze}
For every $\varepsilon > 0$ and $t_0\in \N$, there is an integer $N =
N(\varepsilon, t_0)$ such that for every graph $G$ with $n \geq N$ vertices,
 the graph $G$ has an $\varepsilon$-regular
partition $V_0,V_1,\dots,V_t$ where $t_0\leq t \leq N$. Moreover, such a partition
can be found in time $O(n^\omega)$, where $2\le \omega < 2.373$ is the exponent of
matrix multiplication.
\end{lemma}
It should be noted that $N=N(\varepsilon,t_0)$ has a ``terrible'' dependence in $\varepsilon$~\cite{AlonDukeLefmannRodlYuster}: the iterated logarithm $\log^*N$ (the number of times we have to iterate the logarithm to go from
$N$ to 1) is a polynomial in $1/\varepsilon$ (of degree about 20). The dependence of the running time on $\varepsilon$ is upper bounded by $N$ and comes from the number of times the partition is refined.

The edges between pairs of parts of low or high density will be easy for us to label using the dictionaries of the next section, and we can also label the irregular edges since there are not too many of those. These edges will only contribute to the $o(n)$-term.  
The following two results allow us to show that the remaining edges (those between pairs $(V_i,V_j)$ that are both regular and neither dense nor sparse) contribute exactly the leading term we are aiming for. Recall that $\mathcal{H}(a,b)$ was defined in Section~\ref{sec:hereditary}.
\begin{lemma}[Lemma 10 in \cite{AlonBaloghBollobasMorris}]
\label{lem:rl:reduced_sparse}
Given $\delta > 0$ and $m,r \in \N$, there exist an $\varepsilon_0> 0$ and $n_0 \in \N$ 
such that the following holds. 
Let $G$ be a graph with disjoint vertex sets $W_1\dots,W_r$ such that $|W_i|\geq n_0$ for all $i\in [r]$, and the pair $(W_i,W_j)$ is $\varepsilon_0$-regular and of density at least $\delta$ and at most $1-\delta$ for all distinct $i,j\in [r]$. Then there exist $a,b\in \Z_{\geq 0}$ with $a+b=r$ such that $G$ contains all graphs from $\mathcal{H}(a,b)$ of at most $m$ vertices as induced subgraphs.
\end{lemma}

We will also need the following classical result in extremal graph theory.

\begin{theorem}[Tur\'{a}n's theorem \cite{Turan}]
\label{thm:Turan}
Let $r\in \N$. Any $K_{r+1}$-free graph $G$ on $n$ vertices has at most $\left(1-\tfrac1{r}\right)\cdot \tfrac{n^2}2$ edges.
\end{theorem}

\subsection{Orientations and Bipartite Graphs}

An \emph{orientation} of an undirected graph $G$ is a directed graph obtained from $G$ by choosing one of the two possible directions for each  edge of 
$G$. We use the following result to efficiently label adjacencies in graphs that are not very dense.
\begin{lemma}[Lemma 3.1 in \cite{AlonTarsi}]
\label{lem:orientations_mad}
A graph $G = (V, E)$ has an orientation in which every out-degree is at most $d$ if and only if 
$\max_{H\subseteq G}\frac{|E(H)|}{|V(H)|}\leq d.$
\end{lemma}

For the results of Section \ref{sec:rl:digraphs} we also need an optimal adjacency labelling scheme for bipartite
graphs from~\cite{DulekebaGawrychowskiJanczewski}.
A scheme with labels of the same order, but with a better lower order term can also be deduced
from~\cite{Alon17} or \cite{AlstrupKaplanThorupZwick}, but without an explicit decoder.
\begin{lemma}[Theorem 3.2 in~\cite{DulekebaGawrychowskiJanczewski}]\label{lem:bipartite}
The class of  bipartite graphs has an adjacency labelling scheme with
labels of size at most $n/4+10 \log n$. The labelling can be
constructed in time $O(n^2)$ and adjacency queries can be answered in
constant time in the word RAM model with words of size $\Theta(\log n)$. In fact, the decoding scheme
inspects at most ten $\lceil \log n \rceil$-bit words.
\end{lemma}

\subsection{Succinct Dictionaries for Neighbourhoods}
\label{subsec:dict}
Throughout this paper, $\log$ stands for the logarithm to the base 2,
and the natural logarithm is denoted by $\ln$. 

We will need an encoding of subsets $S$ of $[n]=\{1,\dots,n\}$ from which we can answer membership queries such as ``$x\in S$?'' efficiently.
It is easy to encode a subset of size at most $k$ from $[n]$ by using
$(k+1)\lceil \log n\rceil$ bits to write down the size of the subset
and then the elements in increasing order (each element in base 2, as a $\lceil \log
n\rceil$-bit word). Membership queries can then be performed using binary search. This gives the following folklore result which will suffice for the proof of Theorem \ref{thm:trade_off}. 
\begin{theorem}[Folklore dictionaries] \label{thm:dictionaries}
  For any integer $n\ge 2$ and $k\in [n]$, any subset of $[n]$ of size
  at most $k$ can be encoded in time $O(k\log n)$ using at
  most $(k+1)\lceil \log n\rceil \le 4k\log n$ bits of storage, such that
  membership queries can be answered by inspecting at most $4\log
  k\cdot \log n$ bits.
\end{theorem}
A more efficient scheme has been introduced by Fredman, Koml\'{o}s and
Szemer\'{e}di \cite{FredmanKomlosSzemeredi}, but this would not significantly impact the asymptotics of our decoding time in the proof of Theorem \ref{thm:trade_off} and we therefore opt to use the simplest solution. 

For our other results, we will use a scheme that encodes subsets using a number of bits close to the information-theoretic minimum, while also requiring only constant query time in the classical \textit{word RAM model}~\cite{wordram}. In this model we have access to an array whose cells contain $w$-bit words (or equivalently integers in the interval $[2^w]$). Each word can be accessed in constant time. Moreover, usual arithmetic operations between integers in the interval $[2^w]$ and bitwise logical operations between $w$-bit strings can be performed in constant time as well. The constant decoding time for Theorem \ref{thm:Hfree} and Theorem \ref{thm:constant} are all measured in this model as well as the constant decoding time from the earlier work of Dul{\k{e}}ba, Gawrychowski and Janczewski~\cite{DulekebaGawrychowskiJanczewski} on reachability labelling schemes.

The information-theoretic minimum for encoding subsets of $[n]$ of
size $k=k(n)$ is $\lceil \log \binom{n}{k}\rceil$ bits. Let $H(p)=-p\log(p)-(1-p)\log(1-p)$ denote
the binary entropy function. Then $\log \binom{n}{k} \leq  H(k/n)\cdot
n$ where $H(p)\to 0$ as $p\to 0$. Theorem 1.1 in \cite{pagh}
shows that subsets of size $k$ from $[n]$ can be encoded at the information-theoretic minimum up to lower order terms with constant membership query time in the word RAM model. By adding at most $m$ elements to the universe $[n]$, we can complete all sets of at most $m$ elements of $[n]$ into $m$-element subsets of $[n+m]$.
We obtain the following corollary. 
\begin{theorem}[Efficient static dictionaries \cite{pagh}]
\label{thm:dictionaries_opt}
For any integer $n\geq 2$ and $k\in [n]$, subsets of size at most $k$
from $[n]$ can be encoded in time $O(n^3)$ using at most
\[
2H(k/n) \cdot n + O\left(\frac{k(\log\log k)^2}{\log k}+\log\log n\right)
\]
bits of storage, with constant membership query time in the word RAM model with word size
$\Theta(\log n)$.
\end{theorem}
A similar result was obtained by Brodnik and Munro
in~\cite{BM97}, with a slightly worse lower order term and no explicit
analysis of the time complexity of the encoding.

Our main use of Theorem \ref{thm:dictionaries_opt} is in the following form.
\begin{corollary}
\label{cor:dict_opt}
Let $\varepsilon>0$ be given. For every $n$-vertex graph with an orientation on the vertices 
in which each vertex has out-degree at most $\varepsilon n$, we can encode the adjacencies of the graph in time $O(n^3)$ using labels of at most $2H(\varepsilon) n + O\left(n(\log \log n)^2/\log n\right)$ bits per vertex, with constant adjacency query time in the word RAM model with word size
$\Theta(\log n)$.
\end{corollary}

To see how this follows from Theorem \ref{thm:dictionaries_opt}, note that we can number the vertices of the graph with the elements of $[n]$ and can then let each vertex store the set of its out-neighbours. Deciding whether two vertices $u$ and $v$ are adjacent is equivalent to deciding whether $u$ lies in the out-neighbourhood of $v$ or $v$ lies in the out-neighbourhood of $u$.

\section{Adjacency Labelling in Hereditary Classes}
\label{sec:Hfree}
We show that, for every hereditary class of graphs $\mathcal{G}$, there is an adjacency labelling scheme using $\left(1-\chi_c(\mathcal{G})^{-1}+o(1)\right)\cdot \tfrac{n}2$ bits per vertex. This implies Theorem~\ref{thm:Hfree} by Theorem~\ref{thm:ABT} and Observation~\ref{obs:mu}.
\begin{proof}[Proof of Theorem \ref{thm:Hfree}]
Let $\mathcal{G}$ be a hereditary graph class, with $r=\chi_c(\mathcal{G})\ge 1$. We can assume that $r$ is finite, by using the simple adjacency labelling scheme of Moon~\cite{moon65} with labels of at most  $\tfrac{n}2+\log n$ bits if $\mathcal{G}$ is the class of all graphs (see also~\cite{Alon17,AlstrupKaplanThorupZwick} for more recent work on this case). We will show that there exists a continuous function $f:\R_{>0}\to \R$ such that $f(x)\to 0$ as $x\to 0$ and a function $n_2:\R_{>0}\to \N$, such that for all $\delta\in (0,1)$, for all $G\in \mathcal{G}$ on $n\geq n_2(\delta)$ vertices, we can construct labels of at most $(1-\tfrac1{r}+f(\delta))\cdot \tfrac{n}2$ bits per vertex in time $O(n^3)$, such that it can determined from the labels of two vertices whether they are adjacent.

By definition of the colouring number, for all $a,b\in \Z_{\geq 0}$ such that $a+b=r+1$, there exists a graph $H_{a,b}\in \mathcal{H}(a,b)$ such that $H_{a,b}\not \in \mathcal{G}$. Let $m=\max|V(H_{a,b})|$, where the maximum is taken over the (finite number of) $a,b\in \Z_{\geq 0}$ such that $a+b=r+1$. 

Let $\delta\in (0,1)$ be given and let $\varepsilon_0$ and $n_0$ be given from Lemma \ref{lem:rl:reduced_sparse} applied for the $\delta$ and $m$ defined above and $r=\chi_c(\mathcal{G})$. We set $\varepsilon = \min\{\varepsilon_0,\delta\}$ and $t_0=\lceil \delta^{-1}\rceil$. Let $N=N(\varepsilon,t_0)$ be given from Lemma \ref{lem:rl:alg_sze}. We set $n_1=n_0(1-\varepsilon)^{-1}N$. 

Let $G\in \mathcal{G}$ be an $n$-vertex graph with $n\geq n_1$. We apply Lemma \ref{lem:rl:alg_sze} to $G$ to find an $\varepsilon$-regular partition $(V_0,\dots,V_t)$ of $V$ in time $O(n^\omega)$ with $t_0\leq t+1\leq N$ parts. Note that $N=N(\varepsilon,t_0)$ depends on $\delta$ and $\mathcal{G}$ but does not depend on $G$. By definition of $n_1$ and $t_0$ we find $n_0\leq |V_i|\leq \delta n$ for all $i\in [t]$.
We define a 4-colouring $c$ of the edges of the complete graph $K_t$ with vertex set $[t]$, as follows. 
\begin{itemize}
    \item We colour the edge $ij$ red if $(V_i,V_j)$ is not $\varepsilon$-regular.
    \item We colour the edge $ij$ black if $d(V_i,V_j)>1-\delta$ and $(V_i,V_j)$ is $\varepsilon$-regular.
    \item We colour the edge $ij$ white if $d(V_i,V_j)<\delta$ and $(V_i,V_j)$ is $\varepsilon$-regular.
    \item We colour the edge $ij$ grey if $\delta \leq d(V_i,V_j)\leq 1-\delta$ and $(V_i,V_j)$ is $\varepsilon$-regular.
\end{itemize}
Using a constant number of bits, we store this auxiliary graph of
constant size (including the colour of the edges) in the label of each vertex. We also store the
size of the corresponding parts of $G$ (note that it suffices to store
$|V_0|$ and $n$ since $|V_1|=\dots=|V_t|$, which takes at most
$2\lceil \log n\rceil$ bits).  Recall that $|V_i|\leq \delta n$ for all $i\in\{0,\dots,t\}$. We order the vertices within their own part arbitrarily. For each vertex $v\in G$, we record the index $i$ of the part $V_i$ containing $v$,  as well as the position of $v$ in the order on $V_i$. We then allocate a further $\lfloor \delta n\rfloor$ bits for each vertex, where the $j$-th bit encodes whether $v$ is adjacent to the $j$-th vertex of its part. Furthermore, we allocate an additional $\lfloor \delta n\rfloor$ bits to record the adjacencies from each vertex to the part $V_0$. 

It now remains to encode the adjacencies between $V_i$ and $V_j$ for $1\leq i< j\leq t$. For this, we apply different labelling schemes depending on the colour of the edge $ij$.

\medskip
We first consider the grey edges. 
\begin{claim}
\label{cl:grey}
The subgraph $Y$ of $K_t$ induced by the grey edges can be oriented so
that each vertex in $Y$ has out-degree at most $\left(1-\tfrac1{r} \right)\cdot \tfrac{t}2$.
\end{claim}
\begin{proof}
Suppose that $Y$ contains a copy of $K_{r+1}$ induced on $i_1,\dots,i_{r+1}$ and set $W_j=V_{i_j}$ for $j\in [r+1]$. Since $\varepsilon\leq \varepsilon_0$, the pairs $(W_i,W_j)$ are $\varepsilon_0$-regular for all distinct $i,j\in [r+1]$. By Lemma \ref{lem:rl:reduced_sparse}, we find that there exist $a,b\in \Z_{\geq 0}$ with $a+b=r+1$ such that $G$ contains an induced copy of every graph in $\mathcal{H}(a,b)$ on at most $m$ vertices. In particular $G$ contains an induced copy of $H_{a,b}$. Since $\mathcal{G}$ is hereditary, we find $H_{a,b}\in \mathcal{G}$, a contradiction.

Hence $Y$ is $K_{r+1}$-free. Tur\'{a}n's theorem (Theorem \ref{thm:Turan}) shows that for every subgraph $Y'$ of $Y$,
$|E(Y')|\leq \left(1-\tfrac1{r}\right)|V(Y')|^2/2$. The claim now follows from Lemma  \ref{lem:orientations_mad}.
\end{proof}
Let $T$ be an orientation of the edges of $Y$ such that every vertex has out-degree at most $\left(1-\tfrac1{r} \right)\tfrac{t}2$. We now encode the corresponding adjacencies in $G$ as follows. For each vertex $v\in V_i$ (with $i\in [t]$), we encode the set of elements $j\in [t]$ for which $i$ is oriented from $i$ to $j$ in $T$. We then encode the adjacencies between $v\in V_i$ and $V_j$ naively by appending a bit string of length $|V_j|\le n/t$ to its label, where (as before) the $k$-th bit in the string indicates whether or not $v$ is adjacent to the $k$-th vertex of $V_j$. By the bound on the out-degrees in the orientation $T$ of $Y$, this adds a total of at most 
\[
\left(1-\frac1{r}\right)\cdot \frac{t}2\cdot \frac{n}{t}=\left(1-\frac1{r}\right)\frac{n}2
\]
bits to each label, plus a constant number of bits for storing the corresponding
out-neighbourhood of $i$ in $T$.

\medskip

We now take care of the red edges.
\begin{claim}
The subgraph $R$ of $K_t$ induced by the red edges can be oriented so
that each vertex in $R$ has  out-degree at most $\sqrt{\varepsilon}
\cdot t$.
\end{claim}
\begin{proof}
It is enough to show that $R$ is $(\sqrt{\varepsilon}
\cdot t)$-degenerate, i.e.\ $R$ has a vertex ordering such that each
vertex has at most $\sqrt{\varepsilon}
\cdot t$ neighbours preceding it in the order (note that given such
ordering, orienting all edges from successors to predecessors gives the
desired orientation). If $R$ is not $(\sqrt{\varepsilon}
\cdot t)$-degenerate, then it contains a subgraph $R'$ of minimum degree
more than $\sqrt{\varepsilon}
\cdot t$, and thus with more than $\sqrt{\varepsilon}
\cdot t$ vertices. It follows that $R'$ contains at least
$\tfrac12(\sqrt{\varepsilon} \cdot t)^2=\varepsilon \tfrac{t^2}2$ edges,
contradicting the fact that $R'$, as a subgraph of $R$, contains at most $\varepsilon {t \choose 2}$ edges.
\end{proof}
By the claim, we can encode the edges in $G$ corresponding to the red
edges using at most $\sqrt{ \varepsilon}\cdot n+O(\log n)$ bits as in the case of the grey edges. Indeed, we fix an orientation satisfying the claim, and use a
constant number of bits to record the orientation. For each $x\in V_i$
and each pair $(V_i,V_j)$ such that $ij$ is a red edge that has been
oriented from $V_i$ to $V_j$ in the orientation of $R$ resulting from
the claim above, we record the adjacencies from $x$ to $V_j$ naively
using a bit string of length $n/t\geq |V_j|$. Since the out-degrees in
the red graph $R$ are at most $\sqrt{\varepsilon}\cdot t$, this adds a total of at most $(\sqrt{\varepsilon}
\cdot t)(n/t)=\sqrt{ \varepsilon}\cdot n$ bits to each label.

\medskip

We next turn our attention to the white edges. Recall that a white edge $ij$ corresponds to an $\varepsilon$-regular pair $(V_i,V_j)$ of density at most $\delta$.
\begin{claim}
\label{claim:white}
Let $(V_i,V_j)$ be a pair corresponding to a white edge $ij$. Then we can find in time $O(n^2)$ an orientation of the edges between $V_i$ and $V_j$ so that each vertex has out-degree at most $2\delta n/t$.
\end{claim}
\begin{proof}
As in the previous claim, it is enough to show that the subgraph of
$G$ induced by the edges between $V_i$ and $V_j$ is $(2\delta
n/t)$-degenerate (finding the corresponding vertex ordering can easily be done in time $O(n^2)$).  If it is not, then it contains some nonempty subgraph
with minimum degree more than $2\delta n/t$.  Let $W_i$ and $W_j$ be
the intersection of the vertex set of this subgraph with $V_i$ and
$V_j$, respectively. But then $\min(|W_i|,|W_j|)\ge 2\delta
n/t>\varepsilon |V_i|=\varepsilon |V_j|$ (where we have used $|V_i|\le n/t$ and $\varepsilon\le \delta$) while $d(W_i,W_j)> 2\delta\geq \delta +\varepsilon$.  This contradicts the fact that $(V_i,V_j)$ is an $\varepsilon$-regular pair of density at most $\delta$.
\end{proof}
Let $W$ be the (bipartite) subgraph of $G$ induced by all the edges
connecting sets $V_i$ and $V_j$ such that $ij$ is a white edge. The
claim above implies that $W$ has an orientation such that each vertex
has out-degree at most $t \cdot 2\delta n/t=2 \delta n$. Hence by Corollary \ref{cor:dict_opt}, we can encode the adjacencies of $W$ using at most $2H(2\delta) n+o(n)$ bits.

\medskip

The black edges can be handled similarly, by encoding non-adjacencies
instead of adjacencies. More specifically, we consider the graph
with vertex set $V$ in which we add an edge between $v\in V_i$ and $w\in
V_j$ whenever $ij$ is a black edge and $vw$ are not adjacent in $G$
(in other words, we consider the complement of $G$ and only keep edges
between pairs of sets $V_i,V_j$ such that the edge $ij$ is black). By symmetry,
the analysis above shows that this graph has an orientation such that
each vertex $v$ has out-degree at most $2\delta n$, and we can again apply Corollary \ref{cor:dict_opt}. 
\medskip

In total, we used at most
\[
\left(2\delta+\frac12\left(1-\frac1{r}\right)+\sqrt{\varepsilon} +2 \cdot 2H(2\delta)  +o(1)\right)\cdot n
\]
bits in each label. Note that $\varepsilon\leq \delta$ and that $(1+2\delta)\cdot H(2\delta)\to 0$ as $\delta\to 0$. So we have proved that for $n$ sufficiently large that the labels have at most
\[
\left(1-\frac1{r}+f(\delta)\right)\cdot \frac{n}2=\left(1-\frac1{\chi_c(\mathcal{G})}+f(\delta)\right)\cdot \frac{n}2
\]
bits for a function $f$ such that $f(\delta)\to 0$ when $\delta\to 0$. This completes the analysis of the length of the labels.

For each of the encodings described above, we reserve a block of bits and start the block with the length of the block. This allows us to efficiently navigate to the next block.

\medskip

We now explain how the decoder works. Given vertices $v,w$, the
decoder first reads off the indices $i,j\in \{0,\dots,t\}$ with $v\in
V_i$ and $w\in V_j$. If $i=0$ then the decoder reads the position of
$v$ in the order on $V_0$, and reads off whether $v$ is adjacent to
$w$ from the label of $w$. A similar procedure applies if $j=0$ or if $i=j$. So we may assume $i>j\geq 1$. We read off the colour of the edge $ij$ in the auxiliary graph $K_t$ from the label of either of the two vertices, and depending on the colour we read different parts of the labels again. 

For grey or red edges, we read off the orientation of the grey or red subgraphs respectively to determine whether $v$ encoded the adjacency $vw$ or vice versa, and can then find the desired bit in the string of either $v$ or $w$ that encodes whether there is an edge between $v$ and $w$. For the black or white edges, we jump to the part of the label of $v$ that encodes the adjacencies (or non-adjacencies) between $v$ and the sets $V_j$ such that $ij$ is a white or black edge. The bit string stored there is used to do a membership test as
to whether $w$ is part of the special set stored for $v$ corresponding
to its out-neighbourhood in the orientation of Claim \ref{claim:white}. If the edge $ij$ is white we need a positive answer to the membership test, while if $ij$ is black we want a negative answer. This can be done in constant time by Corollary \ref{cor:dict_opt}. We then repeat this procedure with $v$ and $w$ switched in order to determine whether $v$ and $w$ are adjacent.

\smallskip

Finally, we describe how to obtain an `almost\footnote{The constant in the running time in Lemma \ref{lem:rl:alg_sze} depends linearly on $N$ and so the encoder will still take time polynomial in $n$, but the number of bits that are inspected by the decoder is no longer constant (since $N$ now depends on $n$). However, we can make the dependence as good as we like by letting $\delta$ tend to zero more slowly.}
efficient' adjacency labelling scheme using $(1-\frac1r+o(1))n$ bits from the schemes above that depend on $\delta>0$. We encode $n$ into the label of each vertex; the decoder uses this to select a $\delta=\delta(n)$ which tends to zero (implying $f(\delta) =o(n)$), while maintaining $n\geq n_2(\delta(n))$. We can for example choose $\delta(n)$ of the form $C(\log^* n)^{-1/20m}$, where $m$ is a constant depending on the graph class defined in the second paragraph of this proof.
\end{proof}

\section{Adjacency Labelling in Comparability Graphs}
\label{sec:rl:digraphs}

We first give the simple version of our labelling scheme for comparability graphs, which allows
for a trade-off between the number of bits used for the label and the
decoding time (measured here by the number of bits inspected by the decoder).
\begin{theorem}
\label{thm:trade_off}
For any $s=s(n)$, the class of comparability graphs has an adjacency labelling scheme with
labels of at most 
\[
n/4+1000s^{-1/3}n\log^2 n+2s
\]
bits (assuming $n\ge 4$), which can be constructed in time
$O(s^{2/3}n^3)$. Moreover the decoder only needs to inspect at most $2s+1000\log^2 n$ bits.
\end{theorem}
\begin{proof}
Let $G=(V,E)$ be a comparability graph with underlying partial order $(V,<)$. The partial order induces a natural orientation of the edges of $G$, by orienting each edge $uv\in E$ with $u<v$ from $u$ to $v$.
In the remainder of the proof, it will be convenient to mix the notation for oriented graphs ($N^+(x)$ and $N^-(x)$ for the out-neighbourhood and
in-neighbourhood of $x$, for instance) with the notation for posets.
We first take a linear ordering $v_1,v_2,\ldots, v_n$ of $(V,<)$ (i.e.\ an
ordering such that if $v_i<v_j$ then $i<j$) and
let each vertex $v_i$ store its position $i$ in the order (this uses
a single word of $\lceil \log n\rceil$ bits).

Given a pair $x,y$ of vertices of $G$ with $x<y$, we say that a vertex $z$
is \emph{covered} by the pair $x,y$ if $x<z<y$. We say that a pair
$x,y$ of vertices of $G$ with $x<y$  is
\emph{heavy} if the number of vertices covered by the pair $x,y$ is at
least $\gamma n$ for $\gamma=3s^{-1}\ln n$. Note that for every heavy pair $x,y$, we have $xy\in E$. Any pair of vertices $x,y$ of $G$ that is not heavy is said
to be \emph{light} (in particular all pairs of non-adjacent vertices in $G$ are light). Observe that
if $x,y$ is a heavy pair and  $z$ is a vertex of $G$ chosen uniformly at
random, then $z$ is covered by the pair $x,y$ (in the sense defined above) with probability at least $\gamma$. It follows
that for any collection $B$ of
heavy pairs in $G$, there is a vertex $z$ covered by a fraction of at
least $\gamma$ of the heavy pairs of $B$, and such a vertex $z$ can be easily found in time $O(n^2)$. 

We now construct a set $S$ of at most $s$ vertices as follows. Let $B$ be the set
of all heavy pairs in $G$. As long as $B\ne \emptyset$ and $|S|<s$, we find a vertex
$z$ covered by a fraction of at least $\gamma$ of the pairs of $B$, as
described above. We then add $z$ to $S$ and remove from $B$ all the
pairs covering $z$. Each time we add an element to $S$, the size of
$B$ is multiplied by a factor at most $1-\gamma\le \exp(-\gamma)$, so
there is no heavy pair remaining after this procedure, since
$\exp(-\gamma s) n^2=\exp(-3\ln n) n^2< 1$. So, the set $S$ has the property that
for any heavy pair $x,y$, there is a vertex $z\in S$ such that $x<z<y$
or $y<z<x$ (i.e.\ the heavy pair $x,y$ covers $z$). Moreover, $S$ can be constructed in time $O(sn^2)$.

We now arbitrarily order the elements of $S$ as $u_1,\ldots,u_s$, and
each vertex $u_i$ of $S$ stores its index $i$. Each vertex $v$ of $G$
stores a bit string of $2|S|=2s$
bits such that the
$(2j-1)$-th bit of the string tells whether $v<u_j$ and the $2j$-th whether $v>u_j$.
For any pair $x,y$, we can find whether there is a vertex $z\in S$
such that $x<z<y$ by inspecting at most $2s$ bits. This is enough to determine
adjacencies for all heavy pairs, as for any heavy pair $x,y$, there is
an edge between $x$ and $y$ if and only if there is a
vertex $z\in S$ 
such that $x<z<y$ or $y<z<x$ (and such a vertex $z$ can be identified by inspecting the
labels of $x$ and $y$). 

\smallskip

It remains to determine adjacencies between light pairs of vertices.
We consider the subgraph $G_0$ of $G$ induced by all edges $xy$ such
that $x,y$ forms a light pair. As explained above, it remains to find an adjacency labelling of this
subgraph. In this paragraph we fix a vertex $x$ in $G_0$, and let $G_0^-(x)$ denote the subgraph of $G_0$ induced by the in-neighbourhood $N^-(x)$ in $G_0$. Note that each vertex
$y$ of $G_0^-(x)$ has out-degree at most $\gamma n$ in $G_0^-(x)$, since
otherwise $x,y$ would form a heavy pair, and thus would not be adjacent
in $G_0$. 
We claim that the adjacencies of $G_0^-(x)$ can be labelled using at most $4\gamma n \log n\leq 10 s^{-1} n\log^2n$ bits in the label of each vertex of $G_0^-(x)$.
Each vertex stores the set of indices $i$ of its out-neighbours $v_i$ in $G_0^-(x)$ using at most $4 \gamma n \log n$ bits via the scheme of Theorem~\ref{thm:dictionaries} (using $k=\lfloor \gamma n\rfloor$). This reduces determining whether two vertices $x,y$ are
adjacent in $G_0^-(x)$ to two membership queries: ``is $y$ in the out-neighbourhood of $x$ in this graph?'' and ``is $x$ in the out-neighbourhood of $y$ in this graph?''). The two membership queries require reading at most $8 \log (\gamma n)\cdot
\log n\leq 8 \log^2n$ bits and the encoding can be done in time $O(n\log n)$.

Similarly, if we denote by $G_0^+(x)$ the subgraph of $G_0$ induced by the out-neighbourhood $N^+(x)$ in $G_0$, then we can encode the adjacencies in $G_0^+(x)$ efficiently because each vertex
of $G_0^+(x)$ has in-degree at most $\gamma n$ in $G_0^+(x)$. 
Later in the proof, we will use this scheme for a handful of
well-chosen vertices $x$ of $G_0$, with the goal of covering most of
the light edges of $G$ with few graphs $G_0^-(x)$ and $G_0^+(x)$.

\smallskip

We say that a vertex $z$ of $V$ is \emph{popular} if it has in-degree
and out-degree at least $\delta n$ in $G$ with $\delta=s^{-1/3}\ln n$. Otherwise $z$ is said to be a
\emph{unpopular}. Recall that a pair $x,y$ \emph{covers} $z$ if
$x<z<y$ or $y<z<x$.
Observe that if we take a pair of vertices $x,y$ uniformly at random in $V$, then
any given popular vertex $z$ is covered by the random pair $x,y$ with probability at least $\delta^2=s^{-2/3}\ln^2n$. It follows that for any collection of popular vertices, there is a pair $x,y$ in $G$ that
covers a fraction of at least $\delta^2$ of the vertices in the
collection, and such a pair can be found in time $O(n^3)$. Hence we
may find a set $T$ of vertices of size at most $2t$, with $t=s^{2/3}$, in time
$O(n^3s^{2/3})$, such that all popular vertices are covered by some pair of $T$, since 
\[
(1-\delta^2)^{t}n\le \exp(-t\delta^2)n\le \exp(-s^{2/3}s^{-2/3}\ln^2n)n<1.
\]
For each vertex $x\in T$ we store the adjacencies inside $G_0^+(x)$ and
$G_0^-(x)$ as explained above, at a total cost of $4t\cdot 10 s^{-1} n\log^2
n=40s^{-1/3}n\log^2 n$ bits in the label of each vertex of
$V$.

For each popular vertex $z$, we encode a pair $x,y$ in $T$ with $x<z<y$ in the
label of $z$. That uses at most $2\lceil \log n\rceil\le 4\log n$ bits. For any neighbour $v$ of
$z$ in $G_0$, either $v>z$ in which case $z,v\in G_0^+(x)$, or $v<z$
in which case $z,v\in G_0^-(y)$. Hence we can hence determine whether a popular vertex
$z$ and some vertex $v$ are adjacent in $G_0$ by inspecting the adjacency
labelling schemes corresponding to $G_0^-(x)$, $G_0^+(x)$, $G_0^-(y)$
and $G_0^+(y)$, where $x,y$ is a heavy pair of $T$ (encoded in the
label of $z$) that covers $z$. This last step requires reading at most
$4\cdot 8\log^2 n$
bits in the labels of $z$ and $v$.

It remains to handle the adjacencies between unpopular vertices. Let $\ell=\lfloor s^{-1/3} n\ln n\rfloor$. By definition of unpopular, in the subgraph of $G$ induced by these vertices, all vertices have out-degree at most $\ell$ or in-degree at most $\ell$.
Let $V^+$ be the set of vertices of out-degree at
most $\ell$, and let $V^-$ be the set of vertices of $V(G)\setminus
V^+$ of in-degree at most $\ell$. By the same argument as we used to label the adjacencies of $G_0^-$ above, we can record the adjacencies in the subgraph of $G$ induced by $V^+$ and in the subgraph induced by $V^-$ at a cost of at most $4\ell \log n$ bits in the label of each vertex using Theorem \ref{thm:dictionaries}. The subgraph $B$ induced by the edges between $V^+$ and $V^-$ is bipartite, so
we can record the adjacencies in this subgraph at a cost of $n/4+10\log n$ bits in the label of each vertex by Lemma \ref{lem:bipartite}. When deciding if $u$ is adjacent to $v$ from the labels, we first read their index in the linear order and then at most $16\log \ell \cdot \log n$ bits to see if one is in the in- or out-neighbourhood of the other. If this is not the case, we then inspect a further at most $20\log n$ bits (by Lemma \ref{lem:bipartite}) to verify if the vertices are adjacent in $B$. The vertices $u$ and $v$ are adjacent if and only if at least one of these two tests has a positive answer.

We can hence label the adjacencies between the unpopular vertices at a cost
of at most $n/4+20s^{-1/3}n\log^2 n$ bits in the label of each vertex,
and the decoder for the adjacencies in this subgraph inspects at most
$40 \log(s^{-1/3} n\ln n)\log n\le 100 \log^2 n$ bits of the labels.

\medskip

In total, each label has at most
\[
n/4+(40+20)s^{-1/3} n\log^2n+2s+100\log n 
\]
bits. Moreover the labelling can be constructed in time $O(n^3 s^{2/3}+sn^2)=O(n^3 s^{2/3})$ and the decoding can be done by inspecting at most $2s+1000\log^2 n$ bits.
\end{proof}
Setting $s^{1/3}=n^{1/4}$, Theorem \ref{thm:trade_off} gives labels of $n/4+O(n^{3/4}\log^2 n)$ bits with encoding time $O(n^{7/2})$ and decoding time $O(n^{3/4})$. With $s=\log^9 n$, we obtain labels of $n/4+O(n/\log n)$ bits with encoding time $O(n^{3}\log^6 n)$ and decoding time $O(\log^9 n)$.

In order to get the constant decoding time required for Theorem~\ref{thm:comparability}, we perform a tighter analysis of the proof above to show the following result. 
\begin{theorem}
\label{thm:constant}
The class of comparability graphs admits an adjacency labelling scheme with labels of 
\[
n/4+O\left(\frac{n(\log \log n)^2}{\log^{1/4} n}\right)
\]
bits, with constant query time in the word-RAM model with words of size $\Theta(\log n)$. The encoding can be done in time $O(n^4)$.
\end{theorem}
\begin{proof}[Sketch]
We follow the proof of Theorem \ref{thm:trade_off}. When creating the set $S$ of $s$ vertices, since we do not necessarily make $S$ large enough, there may be some heavy pairs that do not cover any vertex of $S$. Let $B'\subseteq B$ denote those pairs. We set $s=\lceil -2\ln(\gamma)/\gamma \rceil$ so that 
\[
|B'|\leq \exp(-\gamma s)n^2\leq \gamma^2 n^2.
\]
Let $G_1$ be the (undirected) subgraph of $G$ induced by all edges $uv$ such that $u,v$ is a heavy pair from $B'$. Then the vertices of $G_1$ can be ordered greedily as $v_1,\ldots,v_n$ such that each vertex has at most $2\gamma n$ neighbours
preceding it in the order.
Indeed, if a subgraph of $G_1$ had minimum degree at least $2\gamma
n$, then it would have at least $2\gamma n$ vertices and hence $G_1$ would have at least $\tfrac12(2\gamma n)^2>(\gamma n)^2$ edges, contradicting the fact that $|B'|\leq (\gamma n)^2$. This gives us an orientation of $G_1$ in which each vertex has out-degree at most $2\gamma n$. Using Corollary \ref{cor:dict_opt}, we find an adjacency labelling for $G_1$ with labels of size 
\[
4H(2\gamma)n+ O(n(\log \log n)^2/\log n)
\]
in time $O(n^2)$ with constant query time. 

Similarly, when constructing the set $T$ of size $t$, we might have some popular vertices $z$ remaining that do not have a pair $x,y\in T$ with $x<z<y$. We set $t=\lceil -\ln(\delta)/\delta^2 \rceil$ so that this set $R$ of remaining popular vertices has size at most
\[
|R|\leq \exp(-\delta^2 t)n\leq \delta n.
\]
We apply Corollary \ref{cor:dict_opt} to encode the neighbourhoods $N_G(v)\cap R$ for each vertex $v$ of $G$ using at most $2H(\delta)n+O(n(\log \log n)^2/\log n)$ bits.

For $x\in T$, the out- and in-degree in $G_0^-(x)$ and $G_0^+(x)$ respectively are bounded by $\gamma n$. When encoding the adjacencies of the graphs $G_0^-(x)$ and $G_0^+(x)$, we use Corollary \ref{cor:dict_opt} instead of Theorem \ref{thm:dictionaries} to encode the neighbourhoods. 
The graph induced by the unpopular vertices has for each vertex either the in- or the out-degree bounded by $\delta n$. To encode the adjacencies in these graphs, we again apply Corollary \ref{cor:dict_opt}. In total, we used at most
\begin{equation*}
n/4+4(H(2\gamma)+tH(\gamma))n+(2+4)H(\delta)n+2s+O(n(\log \log n)^2/\log n)
\end{equation*}
bits. It can be checked that $H(p)=-p\log p -(1-p)\log(1-p) \leq -2p\log p$ for
$p\in (0,\frac12)$. The result now follows by setting
\begin{align*}
\delta &= \log^{-1/4} n ,\\
\gamma &= \log^{-3/4} n,\\
t&=\lceil -\ln(\delta)/\delta^2 \rceil=O(\log^{1/2} n \cdot \log \log n),\\
s&=\lceil -2\ln(\gamma)/\gamma \rceil=O(\log^{3/4} n\cdot \log \log n).
\end{align*}
All the parts that were encoded using the dictionaries from
Corollary~\ref{cor:dict_opt} can be decoded in constant time in the
word RAM model. We also need to check for every pair of vertices $x,y$
whether there is a vertex $z\in S$ with $x<z<y$. For each $x$, we create two
strings $s_x^+$ and $s_x^-$ of length $|S|=s=O(\log n)$ that record
whether the $i$-th vertex of $S$ is in $N^+(x)$ or $N^-(x)$
respectively. Given two vertices $x$ and $y$, we can find whether
there is a $z\in S$ with $x<z<y$ by computing the bitwise AND-function on $s_x^+$ and $s_y^-$. These two strings correspond to a constant number of words and therefore this can be done in constant time in the word RAM model with words of $\Theta(\log n)$ bits.
\end{proof}

\section{Conclusion}\label{sec:ccl}
Our main result shows that for all hereditary graph classes $\mathcal{G}$, there exists an adjacency labelling scheme using $\mu_{\mathcal{G}}(n)+o(n)$ bits per vertex, which is optimal up to the $o(n)$ term. For hereditary graph classes with $\chi_c(\mathcal{G})=1$, the leading term in our label size is $o(n)$; for those graphs classes the following problem remains open.
\begin{problem}
Is it true that every hereditary family of graphs $\mathcal{G}$ with $\mu_{\mathcal{G}}(n)=\Omega(\log n)$ has an adjacency labelling scheme using labels of size $(1+o(1))\mu_{\mathcal{G}}(n)$?
\end{problem}
We remark that it is already difficult to obtain the correct leading term for the size of $\mathcal{G}_n$ when $\chi_c(\mathcal{G})=1$. In the range $\log n\le \mu_\mathcal{G}(n) = o(n)$, it was proved in~\cite{BBW01} that the behaviour of  $\mu_\mathcal{G}(n)$ when $n\to \infty$ can be fairly erratic, with examples showing that the function $\mu_\mathcal{G}(n)$ can oscillate between the two extremes of the range as $n$ grows. However, it might be the case that for some specific hereditary classes containing $2^{o(n^2)}$ $n$-vertex graphs, labelling schemes using labels of size $(1+o(1))\mu_{\mathcal{G}}(n)$ can be obtained (for instance using sparse versions of the regularity lemma).

\medskip

We have shown that we can represent a poset using labels of $n/4+O(n^{3/4}\cdot \log^2n)$ bits per vertex.  Kleitman and Rothschild \cite{KleitmanRothschild} showed that there are at most $2^{n^2/4+Cn^{3/2}\log n}$ posets for some constant $C$ and hence an improvement of the lower order term beyond $O(\sqrt{n})$ would be especially interesting. In order to get the better lower order term, we did need to sacrifice the decoding time significantly. We also leave open whether it is possible to get the best of both worlds: can we get a reachability labelling scheme for digraphs with labels of $n/4+O(n^c)$ bits for some constant $c<1$ with constant query time in the word RAM model?

\medskip

We conclude with a nice related problem. It is known that
planar graphs have an adjacency labelling scheme with labels of
$O(\log n)$ bits since the work of Muller~\cite{Mul88} and Kannan, Naor and
Rudich~\cite{KNR88,KNR92} (and it was proved recently that labels of
$(1+o(1))\log n$ bits can
indeed be obtained, close to the lower bound of $\log
n$ bits~\cite{DEJGMM20}). A natural question is whether planar digraphs
also have a
reachability labelling scheme with labels of $O(\log n)$ bits. This is
related to the question of whether planar posets have constant boolean
dimension (see~\cite{FMM20}). The best known reachability labelling
scheme for planar digraphs is due to Thorup~\cite{Tho04}; it uses
labels of $O(\log^2 n)$ bits.

\paragraph{Update May 7, 2021} 
In an earlier version of the paper, we claimed the existence of a universal poset on $2^{(1+o(1))n/4}$ elements, but we recently found a flaw in our argument. We concluded the existence from our comparability labelling scheme, but although this leads to a universal graph for the class of comparability graphs, this universal graph need not be a comparability graph itself.
Unfortunately, we have to take back the claim and we in particular leave open the question raised by Hamkins~\cite{MOquestion} on the minimal size is of a poset that is universal for all posets of size $n$.

\section*{Acknowledgements}
We are grateful to the referees for some helpful suggestions.

\color{black}
\bibliographystyle{scott}
\bibliography{reachability}

\end{document}